\documentclass[12pt]{article}

\usepackage{amssymb,amsmath,amsfonts,amssymb}
-\marginparwidth \oddsidemargin -4mm \evensidemargin
\oddsidemargin
\usepackage{amssymb}
\advance\hoffset by 5mm

\def\@abssec#1{\vspace{.05in}\footnotesize \parindent .2in
{\bf #1. }\ignorespaces}

\newtheorem{theorem}{Theorem}[section]

\newtheorem{lemma}[theorem]{Lemma}
\newtheorem{proposition}[theorem]{Proposition}
\newtheorem{corollary}[theorem]{Corollary}
\newtheorem{definition}[theorem]{Definition}

\def\th{\theta}
\def\om{\omega}

\def \Rm {\mathbb R}

\def \Tm {\mathbb T}

\title{ Regularity and
Blow up for Active Scalars}
\author{Alexander Kiselev\thanks{Department of
Mathematics, University of Wisconsin, Madison, WI 53706, USA;
e-mail: kiselev@math.wisc.edu } }

\begin{document}

\maketitle

\begin{abstract}
We review some recent results for a class of fluid mechanics equations called active
scalars, with fractional dissipation.
Our main examples are the surface quasi-geostrophic equation, the
Burgers equation, and the Cordoba-Cordoba-Fontelos model. We discuss
nonlocal maximum principle methods which allow to prove existence of
global regular solutions for the critical dissipation. We also
recall what is known about the possibility of finite time blow up in
the supercritical regime.
\end{abstract}

\section{Introduction}

In this review, we will be concerned with a class of equations that
are called active scalars (see, e.g. \cite{Const1}):
\begin{equation}\label{as1}
\theta_t = (u \cdot \nabla) \theta -(-\Delta)^\alpha \theta,
\,\,\,\,\theta(x,0)=\theta_0(x).
\end{equation}
Here $\theta$ is a scalar function, and the fluid velocity $u$ is
determined from $\theta$ in a certain way (hence the name "active
scalar"). The parameter $\alpha$ regulates the strength of
dissipation, and in this paper we will consider the range $0 \leq
\alpha < 1.$ Typically, when talking about $\alpha =0$ case, we will
mean that there is no dissipative term in \eqref{as1} (formally,
there should be $-\theta$). The natural generic setting for
\eqref{as1} is either $\Rm^d$ with some decay conditions or
$\Tm^d$ (equivalently, periodic initial data in $\Rm^d$).

The main purpose of this review is to describe recent progress in
understanding active scalars, focusing on nonlocal maximum
principles. This is a technique which appeared for the first time in
\cite{KNV}, and is based on proving that a certain family of moduli
of continuity is preserved by evolution: if the initial data
satisfies certain (nonlocal) condition, then the solution also
satisfies it for all times. So far, this technique has been
particularly effective in handling the critical dissipation strength
for active scalars. It allows to prove existence of global smooth
solutions in this case, where the balance between nonlinear and dissipative
effects is very delicate. One can think of these methods as the sharpest
one can do working with absolute value estimates; any progress from here
to the supercritical regime would likely require insight into finer properties
of dynamics and cancelation. We will discuss the general form of the
nonlocal maximum principles as they apply to active scalars, and
illustrate these general results with some particular examples. Let
us start by setting the stage and writing down several active scalar
equations that will be of interest to us.

Perhaps the best known example of the active scalar is the
two-dimensional Euler equation in the vorticity form. In this case,
$\alpha =0,$ $\theta$ is the vorticity, and $u = \nabla^\perp
(-\Delta)^{-1}\theta$ (here $\nabla^\perp = (\partial_2,
-\partial_1)$). It is well known that the two-dimensional Euler
equation has unique global smooth solutions if the initial data are
sufficiently smooth (see e.g. \cite{MP}).

Another interesting example of an active scalar is the $2D$ surface
quasi-geostrophic (SQG) equation. In this case, $u = \nabla^\perp
(-\Delta)^{-1/2}\theta.$ The SQG equation appears in atmospheric
studies. The starting point is the three dimensional Euler equation
and temperature equation in Boussinesq approximation set in a
strongly rotating half-space. Under certain assumptions, a simpler
system can be used as a model \cite{Held}:
\begin{equation}\label{bousapp}
\Delta \psi =0, \,\,\, \theta = \psi_z,\,\,\,\theta_t = (u \cdot
\nabla) \theta, \,\,\,u = \nabla^\perp \psi.
\end{equation}
Here $z$ is the vertical axis of rotation, $\psi$ is the
streamfunction, $\theta$ is the potential temperature, and all
equations hold in $\Rm^3_+.$ Now on the boundary of the half-space
harmonic function $\psi$ satisfies $\psi_z = (-\Delta_{x,y})^{1/2}
\psi.$ Therefore, the vector field $u$ on the boundary is given by
$u = \nabla^\perp (-\Delta)^{-1/2}\theta.$ This closes the equation
for $\theta$ on the boundary, and the function $\psi$ in the
half-space can be recovered from the boundary values. Fractional dissipation
of power $1/2$ also appears naturally in \eqref{bousapp}: the term
$-(-\Delta)^{1/2}\theta$ in right hand side of the equation for
$\theta$ models Ekman pumping effect in the boundary layer.

In mathematical literature, the SQG equation was first considered by Constantin, Majda and Tabak in \cite{CMT}
(in the conservative case $\alpha =0$). A scenario for finite time blow up, a closing saddle,
was proposed and numerically investigated there. It was later proved by Cordoba \cite{Cord} that
blow up does not happen in this scenario.
Lately, the $2D$ SQG equation attracted much attention from various
authors (see e.g.
\cite{CV,CarFer,ChLee,C,CCW,CMZ,CW,CW11,CW12,dong,DoDu,DoPa,FPV,Ju2,Ju6,Ju7,KNV,Miura,Wu,Wu1,Wu2,Wu3}
where more references can be found). Mainly it is due to the fact
that it is probably the simplest evolutionary fluid dynamics
equation for which the problem of existence of smooth global
solutions remains unsolved (when dissipation is not strong enough - namely, $\alpha<1/2$.

More generally, one can consider a whole spectrum of active scalars
which interpolate between the SQG and Euler equations (see e.g. \cite{CIWu2, LSJS}). The
vector field $u$ in this case is given by
\begin{equation}\label{gensqg} u = \nabla^\perp
(-\Delta)^{-\beta} \theta,\,\,\, 1/2 \leq \beta \leq 1.\end{equation} 
We will call this model $\beta$-generalized SQG.
If $\alpha=0,$ the global existence of smooth solutions for sufficiently nice
initial data is known only for $\beta =1$ (this
corresponds to the Euler equation); in general the dissipation
strength needed for global results increases with decrease in
$\beta$ (more precisely, as we will see below, global results are
available for $\alpha \geq 1-\beta$).

The supercritical $\alpha
<1/2$ regime for the SQG equation and more generally $\alpha <
1-\beta$ regime for \eqref{gensqg} remains little understood. One
has local existence of smooth solutions, but it is not known if
finite time blow up starting from smooth initial data is possible.
The only available global regularity results assume smallness of the initial data in a
certain sense or establish global regularity after
sufficiently long time for dissipation very close to critical \cite{Sylv}.

Another example of an active scalar is the model introduced by
Cordoba, Cordoba and Fontelos \cite{CCF}. The equation is set in one
dimension, and $u$ is a scalar function given by the Hilbert
transform of $\theta:$
\begin{equation}\label{ccfas}
u = H \theta, \,\,\, H \theta(x) = P.V. \frac{1}{\pi}\int
\frac{\theta(y)}{x-y}\,dy.
\end{equation}
The equation \eqref{ccfas} is motivated by Birkhoff-Rott equations modeling the
evolution of vortex sheets with surface tension \cite{BLM,M11}.  This equation is easier to deal with due to
one-dimensionality; but its nonlocal nature still makes it highly
nontrivial. In \cite{CCF}, it was proved that finite time blow up is
possible if there is no dissipation; later in \cite{LR1}, Li and
Rodrigo proved that finite time blow up is possible for $0 < \alpha
< 1/4.$ Existence of global regular solutions is known for $\alpha
\geq 1/2.$ It is not known if finite time blow up is possible for $1/4 \leq \alpha < 1/2.$

Finally, perhaps the simplest active scalar is the classical Burgers
equation in one dimension, where $u = \theta.$ The Burgers equation
is one of the most studied and well-understood models with nonlinearity of fluid
mechanics type. It is well known that shocks can form in finite time when $\alpha
=0,$ and that unique global regular solutions exist if $\alpha =1.$
Burgers equation with fractional dissipation attracted attention
only relatively recently. Nevertheless, this is the only active
scalar for which the issue of global existence of smooth solutions
depending on the strength of dissipation is fully understood: if
$\alpha \geq 1/2$ one has global regularity, while for $\alpha <1/2$
finite time blow up is possible \cite{KNS}.

The plan of this review is as follows. In Section~\ref{gennmp}, we
set up nonlocal maximum principles for active scalars in full
generality. In Section~\ref{reg}, we provide examples of
applications to proving global regularity for various active scalars
at critical dissipation level. In Section~\ref{timeapp}, we
discuss application of nonlocal maximum principles towards existence
of solutions with rough initial data, using the  Burgers equation as an
example. In Section~\ref{bubu}, we sketch out
the argument for finite time blow up in supercritical Burgers
equation. In Section~\ref{buccf}, we sketch the argument for blow up
in Cordoba-Cordoba-Fontelos model, which is different from the original
\cite{CCF} proof.  We conclude with a brief discussion of open problems.
Some of the results in Section~\ref{gennmp} and
Section~\ref{reg} appear here for the first time.
Given the nature of this review, however, we often present only
outline of the proof, avoiding complete technical details and
postponing the full treatment to a later publication \cite{K4}.


\vspace*{0.5cm}
\setcounter{equation}{0}
\section{Nonlocal maximum principles for active
scalars}\label{gennmp}

We start by recalling the definition of the modulus of continuity.

\begin{definition}
A function $\omega: \Rm^+ \mapsto \Rm^+$ is called a modulus of
continuity if $\omega(0)=0,$ $\omega$ is continuous, increasing and
concave. We will also require that $\omega$ is piecewise $C^1$ on
$(0,\infty):$ that is, its derivative is continuous apart from perhaps a finite
number of points where one-sided derivatives exist but may not be equal.

We say that a function $f$ obeys modulus of continuity $\omega$ if
$|f(x)-f(y)| < \omega(|x-y|)$ for all $x \ne y.$

We say that the evolution given by \eqref{as1} preserves $\omega$ if
$\theta(x,t)$ obeys $\omega$ for all times $t>0$ provided that the
initial data $\theta_0(x)$ has $\omega.$
\end{definition}

A classical example of a modulus of continuity is
$\omega(\xi)=\xi^\beta,$ $0<\beta<1,$ corresponding to H\"older
classes.

For simplicity, in this paper we will mostly consider the case
where the active scalar equation \eqref{as1} is set on the torus
$\Tm^d$ (or, equivalently, in $\Rm^d$ with periodic initial data).
One exception will be the blow up argument for the CCF equation in Section~\ref{buccf},
where the whole line case is technically simpler.

Now let us state the general form of a nonlocal maximum principle for active scalars.
Suppose that, if $\theta$ has a given modulus of continuity
$\omega,$ then $u$ can be shown to have modulus of continuity
$\Omega_\omega.$ The exact form of $\Omega_\omega$ depends on the
active scalar under consideration. For example, for the Burgers
equation $\Omega_\omega = \omega;$ we will see other examples in the
next section. Also, define
\begin{eqnarray}\label{dtermest}
D_{\alpha,\omega}(\xi) = \\ c_\alpha \left( \int\limits_0^{\xi/2}
\frac{\omega(\xi+2\eta)+\omega(\xi-2\eta)-2\omega(\xi)}{\eta^{1+2\alpha}}\,d\eta
+\int\limits_{\xi/2}^\infty
\frac{\omega(\xi+2\eta)-\omega(2\eta-\xi)-2\omega(\xi)}{\eta^{1+2\alpha}}\,d\eta
\right), \nonumber
\end{eqnarray}
where $c_\alpha$ are certain fixed positive constants to be
described later. The $D_{\alpha,\omega}(\xi)$ expression will appear
in the estimation of the dissipative term contribution, as will
become clear below.

The following theorem \cite{K4} is the main result of this section.

\begin{theorem}\label{mainthm}
Let $\theta(x,t)$ be a periodic smooth solution of \eqref{as1}. Suppose that
$\omega(\xi,t)$ is continuous in $(\xi,t)$ and piecewise $C^1$ in $t$ for each fixed $\xi,$ that for any
fixed $t \geq 0,$ $\omega(\xi,t)$ is a modulus of continuity, and
that $\partial^2_{\xi\xi}\omega(0,t)=-\infty$ for all $t \geq 0.$ Assume that
the initial data $\theta_0(x)$ obeys $\omega(\xi,0) \equiv
\omega_0(\xi).$ Then $\theta(x,t)$ obeys modulus of continuity $\omega(\xi,t)$ for all $t
>0$ provided that $\omega(\xi,t)$ satisfies
\begin{equation}\label{keyineq}
\partial_t \omega(\xi,t) > \Omega_\omega(\xi,t) \partial_\xi
\omega(\xi,t) + D_{\alpha,\omega}(\xi,t)
\end{equation}
for all $\xi, t >0$ such that $\omega(\xi,t) \leq 2\|\theta(x,t)\|_{L^\infty}.$
In \eqref{keyineq}, at the points where $\partial_\xi
\omega(\xi,t)$ ($\partial_t \omega(\xi,t)$) does not exist, the larger (smaller)
value of the one-sided derivative should be taken.
\end{theorem}
\noindent \it Remark. \rm Of course Theorem~\ref{mainthm} would look nicer if we just assumed that
$\omega$ is smooth away from zero. But in applications, it is often convenient to take $\omega$
with a jump in the first derivative at one point. This is by no means necessary, but it handily simplifies
the estimates. On the balance, it is useful to have Theorem~\ref{mainthm} stated in this more general form.\\
2. The condition that \eqref{keyineq} holds only for $\xi,t$ for which
$\omega(\xi,t) \leq 2\|\theta(x,t)\|_{L^\infty}$ is also natural, since other values of $\xi,t$ are not relevant
for the dynamics. This condition is also in principle not crucial and can be instead addressed by modifying $\omega$ in an application at
hand.  \\

Thus the regularity properties of an active scalar are related to
supersolutions of a strongly nonlinear Burgers-type equation
\eqref{keyineq}, with key terms determined by the nature of vector
field $u$ and strength of dissipation. Dissipation terms which are
more general than $(-\Delta)^\alpha$ can also be studied.

The first step towards the proof of Theorem~\ref{mainthm} is the
following lemma, identifying the scenario how a modulus of
continuity may be lost.

\begin{lemma}\label{scenlemma}
Under conditions of Theorem~\ref{mainthm}, suppose that for some
$t>0$ the solution $\theta(x,t)$ does not obey $\omega(\xi,t).$ Then
there must exist $t_1 >0$ and $x \ne y$ such that for all $t <t_1,$
$\theta(x,t)$ obeys $\omega(\xi,t),$ while
\[ \theta(x,t_1) - \theta(y,t_1) = \omega(|x-y|,t_1). \]
\end{lemma}
The proof of this Lemma is similar to the proof presented in Section
3 of \cite{KNS}. Due to the smoothness of solution and compactness
of the domain, the only issue that one has to contend with is the
possibility that at the breakthrough time $t_1,$ we have
\begin{equation}\label{singpoint}
\|\nabla \theta(x,t_1)\|_\infty = \partial_\xi \omega(0,t_1)
\end{equation} - a single point breakthrough. This could only happen if
$\omega$ is Lipshitz at $\xi=0.$ But even in that case, one can show
that \eqref{singpoint} cannot happen at $t=t_1$ due to the
assumption $\partial^2_{\xi\xi}\omega(0,t_1)=-\infty.$ If \eqref{singpoint} held at $t_1,$
it would have implied, by mean value theorem, that the modulus of continuity is already strictly violated somewhere in a small
neighborhood of $x.$

Next, fix the breakthrough time $t_1$ and points $x,y$ as in
Lemma~\ref{scenlemma}, set $\xi \equiv |x-y|$ and consider
\begin{eqnarray} \nonumber
\left. \partial_t \left( \frac{\theta(x,t)-\theta(y,t)}{\omega(\xi,t)}\right)\right|_{t=t_1}= \\
\frac{(u \cdot \nabla)\theta(x,t_1) - (u \cdot \nabla)\theta(y,t_1)-
(-\Delta)^\alpha \theta(x,t_1) +(-\Delta)^\alpha
\theta(y,t_1)-\partial_t \omega(\xi,t_1)}{\omega(\xi,t_1)}
\label{disst}
\end{eqnarray}
We used $\theta(x,t_1)-\theta(y,t_1)=\omega(\xi,t_1)$ to obtain the
last term in \eqref{disst}. The contribution of the first two terms
in \eqref{disst} is the flow contribution and can be expected to potentially
make the solution less regular. The contribution of the third and
fourth terms in \eqref{disst} is the dissipation contribution and
should work in favor of regularity.

\begin{lemma}\label{flowlem1}
The flow term contribution in \eqref{disst} can be estimated from
above by
\begin{equation}\label{floweq11}
\Omega_\omega(\xi,t_1)\partial_\xi \omega(\xi,t_1),
\end{equation}
with the larger of the one-sided derivatives taken in \eqref{floweq11}
if  $\partial_\xi \omega(\xi,t_1)$ does not exist.
\end{lemma}
This result follows from our assumption that $u$ must have
$\Omega_\omega$ if $\theta$ has $\omega.$ Indeed,
\[ (u \cdot \nabla)\theta(x,t_1) = \lim_{h \rightarrow 0}\frac1h
\left(\theta(x+u(x)h,t_1)-\theta(x,t_1)\right) \] and a similar
representation holds at the point $y.$
Using $\theta(x,t_1)-\theta(y,t_1) = \omega(\xi,t_1)$ and estimating the other
difference by the modulus of continuity, we obtain
\[ (u \cdot \nabla)\theta(x,t_1)-(u \cdot \nabla)\theta(y,t_1) \leq
\lim_{h \rightarrow 0}\frac1h\left( \omega(\xi + h|u(x)-u(y)|,t_1)-\omega(\xi,t_1) \right)
\leq \Omega_\omega(\xi,t_1)\partial_\xi \omega(\xi,t_1). \]

Observe that the estimate \eqref{floweq11} will not be available in
fluid mechanics equations where there is a nonlocal operator in
front of nonlinear term (such as $3D$ Navier-Stokes equations, for
example, after the Leray projection is applied to remove the
pressure term).

Next, we estimate the dissipation contribution.
\begin{lemma}\label{disslem}
The dissipation contribution in \eqref{disst} can be estimated from
above by $D_{\alpha,\omega}(\xi,t_1),$ given by \eqref{dtermest}.
\end{lemma}
\begin{proof}
Let $\Phi^\alpha_{d,t}(x)$ be generalized heat kernel for fractional diffusion in $\Rm^d,$ so that
\[ e^{-(-\Delta)^\alpha t}f(x) = \int_{\Rm^d} \Phi^\alpha_{d,t}(x-y)f(y)\,dy, \]
$0<\alpha<1.$ By scaling, $\Phi^\alpha_t(x)= t^{-d/2\alpha}\Phi^\alpha(t^{-1/2\alpha}x),$
where $\Phi^\alpha$ corresponds to $t=1.$ It is well known that $\Phi^\alpha(x)$ is positive,
spherically symmetric and monotone decreasing in radial variable (see e.g. \cite{Feller}).
Moreover, we have \cite{BSS}
\begin{equation}\label{fracdiffest}
c_\alpha {\rm min}\left( t^{-d/\alpha}, \frac{t}{|x|^{d+2\alpha}} \right) \leq
\Phi^\alpha_{d,t}(x) \leq c_\alpha^{-1} {\rm min}\left( t^{-d/\alpha}, \frac{t}{|x|^{d+2\alpha}} \right).
\end{equation}
for some $0<c_\alpha<1.$
Now the dissipative contribution in \eqref{disst} can be written as (omitting dependence on $t_1$)
\[
-(-\Delta)^\alpha \theta(x) +(-\Delta)^\alpha \theta(y) = \lim_{t \rightarrow 0}
\frac{1}{t} \left((\Phi^\alpha_{d,t} * \theta) (x) - (\Phi^\alpha_{d,t} * \theta) (y) -\omega(|x-y|) \right),
\]
where we used $\theta(x)-\theta(y)=\omega(|x-y|)$ at the breakthrough point. In the above, $\theta$ is extended to all of $\Rm^d$
by periodicity, and convolution is taken over all $\Rm^d.$ This formula holds for periodic functions of all periods and orientations,
so we can freely assume that $x,y$ lie on the $x_1$ axis, and set $x = (\xi/2,0,\dots,0)$ and $y = (-\xi/2,0,\dots,0).$
Write
\begin{eqnarray*} (\Phi^\alpha_{d,t}*\th)(x)-(\Phi^\alpha_{d,t}*\th)(y)=\iint_{\Rm^d}
[\Phi^\alpha_{d,t}(\tfrac{\xi}{2}-\eta,\nu)-\Phi^\alpha_{d,t}(-\tfrac{\xi}{2}-\eta,\nu)]\th(\eta,\nu)\,
d\eta d\nu
\\
=\int_{\Rm^{d-1}}d\nu\int_0^\infty
[\Phi^\alpha_{d,t}(\tfrac{\xi}{2}-\eta,\nu)-\Phi^\alpha_{d,t}(-\tfrac{\xi}{2}-\eta,\nu)]
[\th(\eta,\nu)-\th(-\eta,\nu)]\,d\eta
\\
\le
\int_{\Rm^{d-1}}d\nu\int_0^\infty
[\Phi^\alpha_{d,t}(\tfrac{\xi}{2}-\eta,\nu)-\Phi^\alpha_{d,t}(-\tfrac{\xi}{2}-\eta,\nu)]
\om(2\eta)\,d\eta
\\
=
\int_0^\infty[\Phi^\alpha_{1,t}(\tfrac{\xi}{2}-\eta)-\Phi^\alpha_{1,t}(-\tfrac{\xi}{2}-\eta)]
\om(2\eta)\,d\eta
\\
=
\int_0^\xi \Phi^\alpha_{1,t}(\tfrac{\xi}{2}-\eta)\om(2\eta)\,d\eta+
\int_0^\infty \Phi^\alpha_{1,t}(\tfrac{\xi}{2}+\eta)[\om(2\eta+2\xi)-\om(2\eta)]\,d\eta
\end{eqnarray*}
where $\Phi^\alpha_{1,t}$ is the $1$-dimensional fractional heat kernel, and $\nu$ is $(d-1)-$dimensional. Here we used symmetry and
monotonicity of the fractional heat kernels together with the observation that
$\int_{\Rm^{d-1}}\Phi^\alpha_{d,t}(\eta,\nu)\,d\nu=\Phi^\alpha_{1,t}(\eta)$.
The last formula can also be
rewritten as
$$
\int_0^{\frac{\xi}{2}}\Phi^\alpha_{1,t}(\eta)[\om(\xi+2\eta)+\om(\xi-2\eta)]\,d\eta+
\int_{\frac{\xi}{2}}^\infty
\Phi^\alpha_{1,t}(\eta)[\om(2\eta+\xi)-\om(2\eta-\xi)]\,d\eta\,.
$$
Recalling that $\int_0^\infty \Phi^\alpha_{1,t}(\eta)\,d\eta=\frac{1}{2}$, we see that the
difference $(\Phi^\alpha_{d,t}*\th)(x)-(\Phi^\alpha_{d,t}*\th)(y)-\om(\xi)$ can be estimated from
above by
\begin{eqnarray*}
\int_0^{\frac{\xi}{2}}\Phi^\alpha_{1,t}(\eta)[\om(\xi+2\eta)+\om(\xi-2\eta)-2\om(\xi)]
\,d\eta
\\
+
\int_{\frac{\xi}{2}}^\infty
\Phi^\alpha_{1,t}(\eta)[\om(2\eta+\xi)-\om(2\eta-\xi)-2\om(\xi)]\,d\eta\,.
\end{eqnarray*}
Observe that both expressions in square brackets are negative due to concavity of $\om.$
Then using \eqref{fracdiffest}, dividing by $t$ and passing to $t \to 0+,$ we finally conclude that
the contribution of the dissipative term to our derivative is bounded
from above by
\begin{eqnarray}\label{dissipcont}
c_\alpha \int_0^{\frac{\xi}{2}}\frac{\om(\xi+2\eta)+\om(\xi-2\eta)-2\om(\xi)}{\eta^{1+2\alpha}}
\,d\eta
\\ \nonumber
+
c_\alpha \int_{\frac{\xi}{2}}^\infty
\frac{\om(2\eta+\xi)-\om(2\eta-\xi)-2\om(\xi)}{\eta^{1+2\alpha}}\,d\eta\,\equiv D_{\alpha,\omega}(\xi).
\end{eqnarray}
\end{proof}

Now given Lemma~\ref{flowlem1}, Lemma~\ref{disslem} and \eqref{keyineq}, \eqref{disst} shows that
\[ \left. \partial_t \left( \frac{\theta(x,t)-\theta(y,t)}{\omega(\xi,t)}\right)\right|_{t=t_1} <0. \]
But this is a contradiction with our choice of $t_1,$ $x$ and $y,$ as it implies that the modulus of continuity must have been violated at an
earlier time. This completes the proof of Theorem~\ref{mainthm}.

\vspace*{0.5cm}
\setcounter{equation}{0}
\section{Critical regularity for active scalars: examples}\label{reg}

In this section, we will look at applications of Theorem~\ref{mainthm} to particular cases of active scalars.
All active scalars \eqref{as1} that we consider have existence of local smooth solutions provided that the initial data are sufficiently
smooth. This result can be proved by standard methods, for example using uniform estimates for Galerkin approximations
(see, e.g. \cite{R} or \cite{KNS} for some particular cases that can be extended to other models).
Another property that all our models share is non-increasing $L^p$ norms for smooth solutions. The strongest control of the solution that follows is given by the $L^\infty$ norm.
This control makes $\alpha=1/2$ critical for the SQG, Burgers and CCF models, and $\alpha = 1-\beta$ for the $\beta$-generalized SQG. Larger values of $\alpha$ correspond to the
subcritical regime, and smooth solution exists globally in this case. This can be proved by standard arguments; again, see for example \cite{R} or \cite{KNS}
for the SQG or Burgers respectively with very similar arguments applicable for other models.

In this section, we will describe how to extend global regularity results to the case of critical dissipation in all of our examples. Nonlocal maximum principles play a key role in this
step. Lemma~\ref{disslem} provides us with an estimate for dissipation term $D_{\alpha,\omega}(\xi,t)$ in \eqref{keyineq}. The next lemma will provide
an estimate for the flow term $\Omega_\omega(\xi,t)$ in a natural class of examples.

Let $\hat{f}(k)$ denote the Fourier transform of function $f$ in $\Rm^d.$
Let $D(k)$ be a $d$-dimensional vector where each component is a linear function of $(k_1,\dots,k_d).$
\begin{lemma}\label{flowlem}
Suppose that $\hat{u}(k) = D(k)|k|^{-\beta}\hat{\th}(k),$ $1 \leq \beta <2,$ $\th$ and $u$ are periodic functions on $\Rm^d.$
If $\th$ has a modulus of continuity $\om(\xi),$ then $u$ has the modulus of continuity
\begin{equation}\label{ucon}
\Omega_\beta(\xi) = A \left( \int_0^\xi \frac{\om(\eta)}{\eta^{2-\beta}}\,d\eta + \xi \int_\xi^\infty \frac{\om(\eta)}{\eta^{3-\beta}}\,d\eta \right),
\end{equation}
where $A$ is a fixed constant (that may depend only on $D,$ $\beta$ and $d$).
\end{lemma}
\it Remark. \rm The class of active scalars to which this lemma is applicable includes the SQG equation and the active scalars interpolating between SQG and Euler,
but not the Euler itself. It also includes the CCF model.
\begin{proof}
The proof easily reduces to the case where $D(k)$ is scalar and is equal to $k_1.$ Then $D(k)|k|^{-\beta}$ is a singular integral operator
which corresponds to a convolution with a kernel $K(r,\varphi) = r^{-d-1+\beta}\Gamma(\varphi),$ where $(r,\varphi)$ are spherical coordinates and
$\Gamma$ is smooth and has mean zero on a unit sphere. The argument for obtaining \eqref{ucon} is then completely parallel to that sketched in the Appendix of \cite{KNV}.
\end{proof}

Let us now discuss some applications of Theorem~\ref{mainthm}, starting with the simpler ones. \\

\it Burgers equation in one dimension \cite{KNS} \rm
\begin{equation}\label{bur1}
\theta_t = \theta \theta_x - (-\Delta)^\alpha \theta, \,\,\,\theta(x,0)=\theta_0(x).
\end{equation}
For the Burgers equation, the critical dissipation strength is $\alpha = 1/2.$ If $\alpha>1/2,$ the existence of global regular solution can be
established using standard methods. The critical $\alpha =1/2$ was recently resolved in \cite{KNS}, where it was shown that the global regular solution exists and is unique for
the initial data $\theta_0 \in H^{1/2}$ (the solution becomes real analytic for any $t>0$). Also, it was shown in \cite{KNS} that finite time blow up can occur if $\alpha <1/2,$ making the picture complete for the Burgers equation. Let us sketch here how Theorem~\ref{mainthm} can be applied to prove global regularity for the critical $\alpha=1/2$ case.

\begin{theorem}\label{burreg11}
Let the initial data $\theta_0$ belong to $H^{s},$ $s \geq 1/2$ Sobolev space. Then the critical dissipative Burgers equation has a unique global solution $\theta(x,t)
\in C([0,\infty),H^s) \cap L^2([0,\infty), H^{s+1/2})$ which is real analytic
for every $t>0.$ Moreover, if $\th_0$ belongs to the Sobolev space $W^{1,\infty},$ then
\begin{equation}\label{gradburest}
\|\th'(x,t)\|_{L^\infty} \leq  \|\theta'_0\|_{L^\infty} \exp(C\|\theta_0\|_{L^{\infty}})
\end{equation}
for all $t>0.$
\end{theorem}
\it Remark. \rm If $\th_0$ does not belong to $W^{1,\infty},$ then the bound \eqref{gradburest} is still valid if an appropriate factor
depending on time and blowing up at $t=0$ is added.
\begin{proof}
We will construct a stationary modulus of continuity which is conserved by evolution.
Let $K$ be a parameter to be fixed later. Set $\xi_0= \left(\frac{K}{4\pi}\right)^2.$ The modulus of continuity is given by
\begin{equation}\label{burmodcon}
\omega(\xi) = \left\{
\begin{array}{ll}
 \frac{\xi}{1+K\sqrt{\xi}}, & 0<\xi\leq \xi_0 \\
 C_K \log \xi, & \xi > \xi_0.
\end{array} \right.
\end{equation}
Here $C_K$ is chosen so that $\omega$ is continuous at $\xi = \xi_0;$ one can check that $C_K \sim (\log K)^{-1}$
if $K$ is sufficiently
large. One can also check that if $K$ is sufficiently large, then $\omega$ is
concave, with negative and increasing second order derivative on
both intervals in \eqref{burmodcon} (on the first interval,
$\omega''(\xi)= -K(3\xi^{-1/2}+K)/4(1+K\sqrt{\xi})^3$). The first
derivative of $\omega$ may jump at $\xi_0,$ but the left
derivative at $\xi_0$ is $\sim K^{-2},$ while the right derivative
is $\sim K^{-2} (\log K)^{-1}.$ We choose $K$ large enough so that
the left derivative is larger than the right derivative assuring
concavity.

It remains to check the condition \eqref{keyineq}, which in our case reduces to
$\omega(\xi) \partial_\xi \omega(\xi) + D_{1/2,\omega}(\xi) <0$ for all $\xi >0.$
We will actually show a stronger result: that
\begin{equation}\label{burmodconest}
2\omega(\xi) \partial_\xi \omega(\xi) + D_{1/2,\omega}(\xi) <0
\end{equation}
for all $\xi>0.$ This will be useful for us in Section~\ref{timeapp} for studying the solutions with rough initial data.

I. \it The case $\xi \leq \xi_0.$ \rm Using the second order
Taylor formula and the fact that $\omega''$ is negative and
monotone increasing on $[0,\xi],$ we obtain that
\[ \omega(\xi+2\eta)+\omega(\xi-2\eta) \leq
\omega(\xi)+\omega'(\xi)2\eta+\omega(\xi)-\omega'(\xi)2\eta+2\omega''(\xi)\eta^2.
\]
This leads to an estimate
\begin{equation}\label{smallxidiss} \int_0^{\frac{\xi}{2}}\frac{\om(\xi+2\eta)+\om(\xi-2\eta)-2\om(\xi)}{\eta^{2}}
\,d\eta \leq \frac{1}{\pi} \xi \omega''(\xi). \end{equation}
From \eqref{burmodcon}, we find that
\[ \omega(\xi)\omega'(\xi) = \frac{\xi^{1/2}(2\xi^{1/2}
+K\xi)}{2(1+K\xi^{1/2})^3}, \] while
\[ \frac1\pi \xi \omega''(\xi) =
-\frac{K(3\xi^{1/2}+K\xi)}{4\pi(1+K\xi^{1/2})^3}. \] Taking into
account that $\xi \leq \xi_0,$ we find
\begin{equation}\label{for1}
2\omega(\xi)\omega'(\xi) +D_{1/2,\omega}(\xi) \leq 0,
\end{equation}
for any $K$.

II. \it The case $\xi > \xi_0.$ \rm 
Due to concavity, we have
$\omega(\xi+2\eta) \leq \omega(2\eta-\xi)+\omega(2\xi),$ and thus
\[ \int_{\frac{\xi}{2}}^\infty
\frac{\om(2\eta+\xi)-\om(2\eta-\xi)-2\om(\xi)}{\eta^{1+2\alpha}}\,d\eta \leq \frac1\pi \int_{\xi/2}^\infty
\frac{\omega(2\xi)-2\omega(\xi)}{\eta^2}\,d\eta. \] Clearly we
have $\omega(2\xi)\leq \frac32 \omega(\xi)$ for $\xi \geq \xi_0$
provided that $K$ was chosen large enough. In this case, we obtain
\begin{equation}\label{largexidiss} \int_{\frac{\xi}{2}}^\infty
\frac{\om(2\eta+\xi)-\om(2\eta-\xi)-2\om(\xi)}{\eta^{1+2\alpha}}\,d\eta \leq -\frac{\omega(\xi)}{\pi \xi}.\end{equation}
Now it follows from \eqref{burmodcon} that $\omega(\xi)\omega'(\xi) =
C_K^2\xi^{-1}\log \xi,$ while $\xi^{-1}\omega(\xi) =
C_K\xi^{-1}\log \xi.$ Given $C_K \sim (\log K)^{-1}$, it as clear that
\begin{equation}\label{for2}
2\omega(\xi)\omega'(\xi)+D_{1/2,\omega}(\xi) \leq 0,\ \ \ \xi\geq
\xi_0
\end{equation}
if only $K$ was chosen sufficiently large.

Thus, by Theorem~\ref{mainthm}, the modulus of continuity \eqref{burmodcon} is preserved by the evolution. Notice that
we made no assumption on the size of the period, so the result is valid for any period.
Observe the scaling properties of the critical Burgers equation: if $\th(x,t)$ is a solution, then so is $\th(B x, B t)$
for any $B >0.$ Then any rescaled modulus of continuity $\om_B(\xi) = \om(B\xi)$ is also preserved by the evolution.
Now given any sufficiently smooth initial data $\th_0,$ we can find $B$ such that $\th_0$ has $\omega_B.$ This follows from the fact that $\om(\xi)$ is unbounded as
$\xi \rightarrow \infty.$  Given that $\omega'(0)=1,$ preservation of $\omega_B$ by the solution $\th(x,t)$ implies the a-priori bound $\|\theta'(x,t)\|_{L^\infty} \leq B,$
for all $t>0.$ This is much stronger than control of $L^\infty$ norm of $\th(x,t),$ and allows to extend the local smooth solution indefinitely.
Finding the value of $B$ from \eqref{burmodcon}
leads to \eqref{gradburest}.

The minimal initial data regularity condition $\theta_0 \in H^{1/2}$ is needed to ensure existence of local solution (which becomes smooth for $t>0$ is $\alpha>0$).
The arguments allowing to establish this are fairly standard, similarly to the arguments needed to establish uniqueness of the solution. We refer to \cite{KNS} for more details.
\end{proof}

\it The SQG \cite{KNV} and the CCF model. \rm The technical aspects of these two models
for the proof of global regularity in critical case are identical. Similarly to the Burgers equation, the velocity $u$ is a zeroth order Fourier
multiplier of $\th$ on the Fourier side (though this time, unlike Burgers, this is a nontrivial multiplier). Similarly, $L^\infty$ norm control makes $\alpha = 1/2$ critical.

\begin{theorem}\label{SQGCCF}
The critical SQG equation has unique global smooth solution if the initial data $\th_0 \in H^1.$
The critical CCF equation has unique global smooth solution if the initial data $\th_0 \in H^{1/2}.$
Moreover, if $\th_0$ belongs to the Sobolev space $W^{1,\infty},$ then for either equation we have
\begin{equation}\label{gradsqgest}
\|\nabla \th(x,t)\|_{L^\infty} \leq  \|\nabla \theta_0\|_{L^\infty} \exp\exp(C\|\theta_0\|_{L^{\infty}})
\end{equation}
for all $t>0.$
\end{theorem}

Again, we will focus on the global regularity proof assuming smooth $\theta_0.$

The modulus of continuity that is conserved by the evolution in this case is given by
\begin{eqnarray}\label{sqgmodcon}
\omega(\xi) =  \xi - \xi^{3/2}, & 0<\xi \leq \delta \\
\omega'(\xi) =  \frac{\gamma}{\xi(4+\log(\xi/\delta))}, & \xi >\delta, \nonumber
\end{eqnarray}
where $\delta$ and $\gamma$ are parameters that will be chosen below.
On the technical level, the main difference between the SQG-CCF and Burgers is the difference in modulus of continuity of $u$ and $\th.$ The flow $u$ is a Riesz
transform of $\th.$ The singular integral operators like Riesz transform preserve H\"older classes, but our $\omega$ is Lipshitz at $\xi=0.$ In this case, we lose a logarithm,
as is stated in Lemma~\ref{Omom} below. This leads to weaker control over the possible growth of $\|\th\|_{L^\infty}$ in \eqref{gradsqgest}.

\begin{lemma}\label{Omom}
For the modulus of continuity $\omega$ given by \eqref{sqgmodcon}, the modulus of continuity of the vector field $u$ satisfies
\begin{equation}\label{sqgmodconu}
\Omega_\omega(\xi,t) \leq A\omega(\xi)(4+|\log(\xi/\delta)|),
\end{equation}
provided that $\delta>0$ in \eqref{sqgmodcon} is sufficiently small.
\end{lemma}
This result follows from Lemma~\ref{flowlem} and some simple estimates, see \cite{KNV} for more details.

It remains to check \eqref{keyineq}, which in the SQG-CCF case reduces to
\begin{equation}\label{sqgin11}
A\omega(\xi)(4+|\log(\xi/\delta)|) \partial_\xi \omega(\xi) + D_{1/2,\omega}(\xi) <0
\end{equation} for all $\xi >0.$

I. $0<\xi <\delta.$ Similarly to the Burgers case \eqref{smallxidiss}, we have
\[ D_{1/2,\omega}(\xi)\leq \int_0^{\frac{\xi}{2}}\frac{\om(\xi+2\eta)+\om(\xi-2\eta)-2\om(\xi)}{\eta^{2}}
\,d\eta \leq \frac{1}{\pi} \xi \omega''(\xi) = -\frac{3}{4\pi}\xi^{1/2}. \]
On the other hand, the first term in \eqref{sqgin11} does not exceed
$A\xi(4+\log(\delta/\xi)),$ which is clearly smaller in absolute value if $\delta$ was chosen sufficiently small.

II. $\xi \geq \delta.$ From \eqref{sqgmodcon}, it follows that in this range
\[ A\omega(\xi)(4+\log(\xi/\delta)) \partial_\xi \omega(\xi) = \frac{\gamma A \omega(\xi)}{\xi}. \]
On the other hand, if $\gamma$ was chosen sufficiently small, then $\frac32\omega(\xi) > \omega(2\xi),$
and so, due to the concavity of $\omega,$
\[  D_{1/2,\omega}(\xi)\leq \int_{\frac{\xi}{2}}^\infty \frac{\om(\xi+2\eta)+\om(2\eta-\xi)-2\om(\xi)}{\eta^{2}}\,d\eta\leq
\frac{\omega(\xi)}{2} \int_{\frac{\xi}{2}}^\infty \frac{d\eta}{\eta^2} \leq \frac{\omega(\xi)}{\xi}. \]
It is clear that if $\gamma$ were chosen small enough, then the dissipative term dominates, and Theorem~\ref{mainthm} is applicable.

The rest of the proof of Theorem~\ref{SQGCCF} follows the Burgers case. The scaling properties of the critical equations are the same as Burgers,
and conservation of $\omega$ implies conservation of a family of moduli of continuity $\omega_B(\xi) = \omega(B\xi).$ This nonlocal maximum principle is sufficient
for proving the bound \eqref{gradsqgest}, which allows to prove the global regularity.

Observe how sharp the balance is in the second range $\xi \geq \delta.$ The double logarithm growth for $\omega$ results in a weaker bound \eqref{gradsqgest} for the gradient
of the solution than in the Burgers case, but this slow rate of growth is necessary to make the estimates work. Indeed, the flow term
is quadratic in $\omega,$ while the dissipative term is linear. With estimate so tight, allowing any additional growth for $\omega$ would require new ideas.

We note that the global regularity of solutions to critical SQG equation was also proved by Caffarelli and Vasseur \cite{CV} using completely different
method, see also \cite{KN3} for a related third proof. They look at a class of weak solutions of passive drift-fractional diffusion equation, and they show that if the velocity is bounded in BMO uniformly in time, then $L^2$ initial data becomes $C^\sigma$ for some $\sigma >0$ for all $t>0.$ This is stronger than $L^\infty$ control of $\theta,$ and allows to prove global regularity
in the SQG case. For their method to work, it is important that the drift $u$ is divergence free (which does not play any role in the method presented here). The results of \cite{CV} are inspired
by DiGiorgi-type iterative techniques, and are local in nature (showing H\"older regularity in a space-time cylinder if one has control of $u$ in a larger cylinder).

\it The $\beta$-generalized SQG \cite{MX1,K4}. \rm For this model, $\alpha = 1-\beta$ is the critical dissipation exponent. We have
\begin{theorem}\label{betaqgthm}
The $\beta$-generalized critical SQG equation has unique global smooth solution if the initial data $\th_0 \in H^1.$
Moreover, if $\th_0$ belongs to the Sobolev space $W^{1,\infty},$ then
\begin{equation}\label{gradbetaest}
\|\nabla \th(x,t)\|_{L^\infty} \leq  \|\nabla \theta_0\|_{L^\infty} \exp(C\|\theta_0\|_{L^{\infty}})
\end{equation}
for all $t>0.$
\end{theorem}
\it Remark. \rm Notice that the estimate \eqref{gradbetaest} is stronger than that for SQG for any $\beta >1/2,$ and is similar to the Burgers equation.

In this case, one modulus of continuity that can be used is given by
\begin{eqnarray}\label{betasqgmodcon}
\omega(\xi) =  \xi - \xi^{1+\alpha}, & 0<\xi \leq \delta \\
\omega'(\xi) =  \frac{\gamma}{\xi}, & \xi >\delta. \nonumber
\end{eqnarray}

According to Lemma~\ref{flowlem}, the $\Omega_\omega(\xi)$ term appearing in \eqref{keyineq} can be set to
\[ \Omega_\beta(\xi) = A \left( \int_0^\xi \frac{\om(\eta)}{\eta^{2-\beta}}\,d\eta + \xi \int_\xi^\infty \frac{\om(\eta)}{\eta^{3-\beta}}\,d\eta \right). \]
We therefore have to check that
\[ \omega'(\xi)\Omega_\beta(\xi) + D(1-\beta,\xi) \leq 0 \]
for all $\xi >0.$ This can be shown similarly to the two previous cases that we considered; we leave it to the interested reader to finish the proof.
The complete argument will appear in \cite{K4}.

The critical $\beta$-generalized SQG can be handled by Caffarelli-Vasseur inspired methods as well. This has been done in \cite{CIWu2}.

\vspace*{0.5cm}
\setcounter{equation}{0}
\section{Rough initial data and time dependent modulus of continuity}\label{timeapp}

So far, in all examples that we considered it was possible to use a stationary modulus of continuity to derive a nonlocal maximum principle. In this section, we provide
an example of application where time dependence becomes important. This is an application on existence solutions to critical SQG equation with rough initial data. The subcritical case can
be handled as well (and is easier). Here we sketch the most important steps in the proof; complete presentation will appear elsewhere.
The results we describe can also be extended to other active scalars. The Burgers equation case was considered in \cite{KNS}.

Consider the equation
\begin{equation}\label{beq123}
\th_t=(u \cdot \nabla) \th-(-\Delta)^{1/2}\th,\, \,\, u = \nabla^\perp(-\Delta)^{-1/2}\th,\,\,\, \th(x,0)=\th_0(x).
\end{equation}
Then we have
\begin{theorem}\label{thmrough}
Let $\th_0\in L^p$ for some $p\in(1,\infty)$. Then there exists a solution $\th(x,t)$
of the equation \eqref{beq123} such that $\th$ is a smooth function for $t>0$,
\begin{equation}
\|\th(\cdot,t)-\th_0(\cdot)\|_{L^p}\to0\ \ \ \hbox{as}\ \ t\to0;
\end{equation}
\begin{equation}
t^{2/p}\|\th(\cdot,t)\|_{L^\infty}\leq C(\|\th_0\|_{L^p}),\ \ \
0<t\leq1;
\end{equation}
\begin{equation}
F(t)^{-1}\|\th(\cdot,t)\|_{W^1_\infty}\leq C(\|\th_0\|_{L^p}),\ \ \
0<t\leq1;
\end{equation}
Here $F$ is defined below in \eqref{deff}.
\end{theorem}
\it Remarks. \rm 1. It can be shown in fact that the solution $\th(x,t)$ is real analytic for every $t>0$
(see \cite{KNS} for the Burgers case argument which can be extended to SQG in a straightforward way).


2. An interesting open question is
the uniqueness of the solution from Theorem~\ref{thmrough}. Due to
the highly singular nature of estimates as $t$ approaches zero,
the usual uniqueness argument based on some sort of Gronwall
inequality is problematic.

Let us look first at the
approximating equation
\begin{equation}\label{eq1}
\th_t^N=(u^N \cdot \nabla) \th^N-(-\Delta)^{1/2}\th^N,\, \,\, u^N = \nabla^\perp(-\Delta)^{-1/2}\th^N,\,\,\, \th^N(x,0)=\th^N_0(x)
\end{equation}
where $\th^N_0\in C^{\infty}$ and $\|\th_0^N-\th_0\|_{L^p}\to0$ as $N\to\infty$.
%
According to the results of the previous section, the solutions $\th^N(x,t)$ are smooth and global.
Recall also that all $L^p$ norms are non-increasing for smooth solutions of active scalar equations.
We divide our proof of regularity into three steps.

{\it Step I.} Here we prove uniform (in $N$) estimates for the
$L^\infty$ norm.
Put
$$
M_N(t):=\|\th^N(\cdot,t)\|_{L^\infty}.
$$
Fix $t\geq0$. Consider any point $x_0$ where $|\th^N(x_0,t)|=M_N$.
Without loss of generality, we may assume that $x_0=0$ and
$\th^N(0,t)=M_N.$ Then
\begin{equation}\label{equat}
\th^N_t(0,t)=(-(-\Delta)^{1/2}\th^N)(0,t)= C\int\limits_{-\infty}^{\infty}\frac{\th^N(y,t)-M_N}{|y|^{3}}dy,
\end{equation}
with a certain constant $C>0.$
Denote Lebesgue measure of a measurable set $S$ by $m(S).$ Since
the $L^p$ norms of solutions are non-increasing, we have
$$
\|\th^N\|_{L^p}^p\leq C.
$$
Then we obtain that \[ m\left.\left(x \right| |\th^N(x,t)| \geq M_N/2
\right) \leq C2^p M_N^{-p}.\] Therefore the right hand side of
\eqref{equat} does not exceed
$$
-CM_N\int\limits_{L\geq|y|\geq C2^{(p-1)/2}}/M_N^{p/2}|y|^{-2}d|y|.
$$
Here $2L$ is the period. Then
\begin{equation}\label{infin1}
\th^N_t(0,t)< -C_1 M_N^{p/2+1}+C_2 M_N.
\end{equation}
The same bound holds for any point $x_0$ where $M_N$ is attained
and by continuity in some neighborhoods of such points. So, we
have \eqref{infin1} in some open set $U_N$. Due to smoothness of
the approximating solution, away from $U_N$ we have
$$
\max\limits_{x\not\in U_N}|\th^N(x,\tau)|< M_N(\tau)
$$
for every $\tau$ during some period of time $[t,t+\tau_N],$
$\tau_N>0.$ Thus we obtain that
\begin{equation}\label{infin}
\frac{d}{dt}M_N< -C_1 M_N^{p/2+1}+C_2 M_N.
\end{equation}
Solving equation \eqref{infin}, we get the uniform estimate
$$
M_N^{p/2}(t)\leq
\frac{e^{C_2pt}}{M_N^{-p/2}(0)+\frac{C_1}{C_2}(e^{C_2pt}-1)}\leq
\frac{C_2}{C_1(1-e^{-C_2pt})}.
$$
In particular,
\begin{equation}\label{linf}
t^{2/p}\|\th^N\|_{L^\infty}\leq C,\ \ \ t\leq1.
\end{equation}

{\it Step II.} Here we obtain uniform in $N$ estimates on the
approximations $\th^N$ that will imply smoothness of the solution.

Clearly, it is sufficient to work with $t \leq 1.$ Let us define
\[ G(t) = {\rm inf}_{0 \leq \omega(x) \leq Ct^{-2/p}} \frac{\omega(x)}{x},
\]
where $C$ is as in \eqref{linf} and $\omega$ is as in \eqref{burmodcon}. Observe that, since $\omega$ is
concave and increasing, the function $G(t)$ is equal to
$Ct^{-2/p}/\omega^{-1}(Ct^{-2/p}).$ Define also
\begin{equation}\label{deff}
F(t) =\left( \int\limits_0^t G(s)\,ds \right)^{-1}.
\end{equation}

Observe that $F(t) \rightarrow \infty$ as $t \rightarrow 0,$ and therefore every $\theta^N(x,t)$ has $\omega_{F(t)}$
for all $t$ small enough (how small may depend on $N$). We claim that $\omega_{F(t)}(\xi) \equiv \omega(F(t)\xi)$ satisfies \eqref{keyineq}.
By Theorem~\ref{mainthm}, this would imply that the solutions $\th^N(x,t)$ have $\omega_{F(t)}$ for all $t>0,$ providing a bound uniform in $N.$
To check \eqref{keyineq}, observe that all the estimates that led to \eqref{burmodconest} when we were proving conservation of $\omega$ for the critical
Burgers equation apply to $\omega_{F(t)}$ in an identical fashion: all terms that enter have the same scaling in $\xi$. In particular,
\[ 2\omega_{F(t)}(\xi) \partial_\xi (\omega_{F(t)}(\xi)) + D_{1/2, \omega_{F(t)}}(\xi) < 0. \]
This and inequality \eqref{linf} imply that to check applicability of Theorem~\ref{mainthm}, we need to verify that
\begin{equation}\label{est1234} -\omega_{F(t)}(\xi) F(t) \partial_\xi \omega_{F(t)}(\xi) \leq F'(t) \xi \partial_\xi \omega_{F(t)}(\xi) \end{equation}
for all $\xi,t>0$ such that
\begin{equation}\label{linfcon11}
 \omega(F(t)\xi) \leq 2\|\theta(x,t)\|_{L^\infty} \leq
2Ct_1^{-1/p}. \end{equation}
This condition reduces to checking
\begin{equation}\label{rr12}
-\frac{F'(t)}{F(t)^2} \leq
\frac{\omega(F(t)\xi)}{F(t)\xi}
\end{equation}
provided \eqref{linfcon11} holds.
Using the definition of the function $G(t),$ we
obtain the estimate
\[ \frac{\omega(F(t)\xi)}{F(t)\xi} \geq G(t). \]
Thus \eqref{rr12} is satisfied if
\[ \left(\frac1F \right)' \leq G(t), \]
which is correct by definition of $F,$ completing the proof of \eqref{keyineq}.

Since $N$ was arbitrary, it follows that $\th(x,t)$ has the
modulus of continuity $\omega_{F(t)}$ for any $t>0,$ and thus
\begin{equation}
\label{last} F(t)\|\th^N\|_{W^1_\infty}\leq C,\ \ \ t\leq1.
\end{equation}
The higher regularity of the solution can now be shown by standard methods similar to the ones applied
to prove the local existence of the solution for regular initial data; we refer to \cite{KNS} for the details. The estimates one obtains
look like
\begin{equation}\label{last1}
F_n(t)\|\th^N(\cdot,t)\|_{1+\frac{n}{2}}\leq C_n,\ \ \ n\geq1,\ \ \ t\leq1,
\end{equation}
with some functions $F_n$ which can be calculated inductively.
Now, we can choose a subsequence $N_j$  such that $\th^{N_j}\to \th$ as
$N_j\to\infty$ and function $\th$ satisfies differential equation
\eqref{beq123} as well as the bounds \eqref{linf}, \eqref{last},
\eqref{last1} on $(0,1]$. By our earlier results this solution can be extended globally.

{\it Step III.} Proving that the function $\th$ can be chosen
to satisfy the initial condition.

The argument here is fairly standard, if somewhat technical. 
We refer to \cite{KNS}, Section 5, for a similar argument given in a complete detail.

\vspace*{0.5cm}
\setcounter{equation}{0}
\section{Finite time blow up: supercritical Burgers equation}\label{bubu}

The Burgers equation is the only active scalar for which there is complete understanding of the issue
of existence of smooth global solutions depending on the strength of the dissipation $\alpha.$
Theorem~\ref{burreg11} shows that smooth global solutions exist if $\alpha \geq 1/2.$ The following result
completes the picture.

\begin{theorem}\label{burbu}
Assume that $0<\alpha<1/2.$ Then there exists smooth periodic
initial data $\th_0(x)$ such that the solution $\th(x,t)$ of
\eqref{bur1} blows up in $H^s$ for each $s>\frac32 -2\alpha$ in finite time.
\end{theorem}

Of course, the reason one can prove finite time blow up for the fractional dissipative Burgers equation is
that the conservative Burgers equation is very well understood (and is well known to form shocks in finite time).
Let us sketch the proof of Theorem~\ref{burbu}, omitting some technical details. A complete presentation
can be found in \cite{KNS}. The sketch we provide gives rather detailed information about the blow up, tracing
carefully the steepening slope in a forming shock. A simpler proof of finite time blow up based on integral inequalities (somewhat similar in spirit
to the proof we will see for the CCF model in Section~\ref{buccf}) was given in \cite{LD}.

At first, we
are going to produce smooth initial data $\th_0(x)$ which leads to
blow up in finite time in the case where the period $2L$ is large.
After that, we will sketch a simple rescaling argument which gives
the blow up for any (and in particular unit) period.

The proof will be by contradiction. We will fix $L$ and the
initial data, and assume that by time $T=T(\alpha)$ the blow up
does not happen. In particular, this implies that there exists $N$
such that $\|\th(x,t)\|_{C^3} \leq N$ for $0 \leq t \leq T.$ This
will lead to a contradiction. The overall plan of the proof is to
reduce the blow up question for front-like data to the study of a
system of finite difference equations on the properly measured
steepness and size of the solution. To control the solution, the
first tool we need is a time splitting approximation. Namely,
consider a time step $h,$ and let $w(x,t)$ solve
\begin{equation}\label{weq}
w_t = ww_x, \,\,\,w(x,0)=\th_0(x),
\end{equation}
while $v(x,t)$ solves
\begin{equation}\label{veq}
v_t = -(-\Delta)^{\alpha}v, \,\,\,v(x,0)=w(x,h).
\end{equation}
The idea of approximating $\th(x,t)$ with time splitting is fairly
common and goes back to the Trotter formula in the linear case
(see for example \cite{BM}, page 120, and \cite{T3}, page 307, for
some applications of time splitting in nonlinear setting). The
situation in the Burgers case is not completely standard, since the
Burgers equation generally does blow up, and moreover the control
we require is in a rather strong norm.

The solution of the problem \eqref{veq} with the initial data $v_0(x)$
is given by the convolution
\begin{equation}\label{disssol} v(x,t)
= \int_\Rm \Phi_t(x-y) v_0(y)\,dy\ (=e^{-(-\Delta)^\alpha t}v_0(x)),
\end{equation}
Recall that
\begin{equation}\label{Phi1}
\Phi_t(x) = t^{-1/2\alpha} \Phi(t^{-1/2\alpha}x), \,\,\,\Phi(x) =
\frac{1}{2\pi}\int_\Rm \exp (ix\xi-|\xi|^{2\alpha}) \,d\xi,
\end{equation}
and that $\Phi(x)$ is even and $\int \Phi(x)\,dx=1.$  Also,
\begin{equation}\label{Phi}
\Phi(x) >0;\,\,\,x\Phi'(x) \leq 0,\,\,\,\Phi(x) \leq
\frac{K(\alpha)}{1+|x|^{1+2\alpha}},\,\,\,\left|\Phi'(x)\right|
\leq \frac{K(\alpha)}{1+|x|^{2+2\alpha}}.
\end{equation}
These properties are not difficult to prove; see e.g.
\cite{Feller} for some results, in particular positivity (Theorem
XIII.6.1). 
We need the following lemma.

\begin{lemma}
For every $f\in C^{n+1},\ \ n\geq0$,
\begin{equation}\label{expcon}
\|(e^{-(-\Delta)^\alpha t} -1)f\|_{C^n} \leq C(\alpha) t
\|f\|_{C^{n+1}}.
\end{equation}
\end{lemma}
The proof of this result can be obtained by direct and simple estimates, see
\cite{KNS}.

The next lemma provides local solvability for our splitting system.
\begin{lemma}\label{localforh}
Assume $\|\th_0(x)\|_{C^3} \leq N.$ Then for all $h$ small enough,
$v(x,h)$ is $C^3$ and is uniquely defined by \eqref{weq},
\eqref{veq}. Moreover, it suffices to assume $h \leq CN^{-1}$ to
ensure
\begin{equation}\label{wvcon}
\|w(x,t)\|_{C^3}, \|v(x,t)\|_{C^3} \leq 2N
\end{equation}
for $0 \leq t \leq h.$
\end{lemma}
\begin{proof}
Using the characteristics one can explicitly solve equation
\eqref{weq}. We have $w(t,y)=\th_0(x)$, where $x=x(y)$ is such that
\begin{equation}
y=x-\th_0(x)t.
\end{equation}
Now, implicit function theorem and direct computations show that
$\|w(t,\cdot)\|_{C^3}\leq 2\|\th_0\|_{C^3}$ provided that
$\|\th_0\|_{C^3}t\leq c$ for some small constant $c>0$. This proves
the statement of the Lemma for $w$. To prove it for $v$ we just
notice that $v$ is a convolution of the $w(h,x)$ with $\Phi_t(x)$.
Since $\|\Phi_t\|_{L^1}=1$ we obtain that
$\|v(t,\cdot)\|_{C^3}\leq\|w(h,\cdot)\|_{C^3}.$
\end{proof}

The main time
splitting result we require is the following
\begin{proposition}\label{timesplit}
Assume that 
$\|\th_0(x)\|_{C^3} \leq N$ for $0 \leq t \leq T.$ Define $v(x,t)$
by \eqref{weq}, \eqref{veq} with time step $h.$ Then for all $h$
small enough, we have
\[ \|\th(x,h)-v(x,h)\|_{C^1} \leq
C(\alpha,N)h^2.\]
\end{proposition}
\begin{proof}
Since $\|\th_0(x)\|_{C^3} \leq N,$ let us choose $h$ as in
Lemma~\ref{localforh}. Notice that by Duhamel's principle,
\[ \th(x,h) = e^{-(-\Delta)^\alpha h} \th_0(x) + \int\limits_0^h
e^{-(-\Delta)^\alpha (h-s)} (\th(x,s) \th_x(x,s))\,ds, \] while
\[ v(x,h) = e^{-(-\Delta)^\alpha h} \th_0(x) + \int\limits_0^h
e^{-(-\Delta)^\alpha h} (w(x,s) w_x(x,s))\,ds. \]
Then it follows from \eqref{expcon} that
\begin{eqnarray}\nonumber
\|\th(x,h)-v(x,h)\|_{C^1} \leq \int\limits_0^h \| e^{-(-\Delta)^\alpha (h-s)}
(\th(x,s) \th_x(x,s)) - e^{-(-\Delta)^\alpha h} (w(x,s)
w_x(x,s))\|_{C^1}\,ds \leq \\
\nonumber \int\limits_0^h \|\th(x,s) \th_x(x,s) - w(x,s)
w_x(x,s)\|_{C^1} \,ds + \int\limits_0^h
\|\left(e^{-(-\Delta)^\alpha (h-s)}-1\right) \th(x,s)
\th_x(x,s)\|_{C^1} + \\ \nonumber \int\limits_0^h
\|\left(e^{-(-\Delta)^\alpha h}-1\right) w(x,s)
w_x(x,s)\|_{C^1}\,ds \leq \int\limits_0^h \|\th(x,s) \th_x(x,s) -
w(x,s) w_x(x,s)\|_{C^1}\,ds +
\\ \label{uvdiff} C(\alpha)h \int\limits_0^h \left( \|\th(x,s)\th_x(x,s)\|_{C^2}+
\|w(x,s) w_x(x,s)\|_{C^2} \right)\,ds.
\end{eqnarray}
From \eqref{wvcon}, it follows that the last integral does not
exceed $C(\alpha) N^2 h^2.$ To estimate the remaining integral, we
need the following
\begin{lemma}\label{uvrough}
For every $0 \leq s \leq h,$ we have $\|\th(x,s)-w(x,s)\|_{C^2} \leq
C(\alpha)N^2 h.$
\end{lemma}
\begin{proof}
Observe that $g(x,s) \equiv \th(x,s) - w(x,s)$ solves
\[ g_t = g\th_x +wg_x -(-\Delta)^\alpha \th, \,\,\,g(x,0)=0. \]
Thus
\[ g(x,t) = \int\limits_0^t (g\th_x +wg_x -(-\Delta)^\alpha \th)\,ds.
\]
Because of \eqref{wvcon} and the assumption on $\th,$ we have
$\|g\th_x\|_{C^2},\|wg_x\|_{C^2} \leq CN^2,$ and $\|(-\Delta)^\alpha
\th\|_{C^2} \leq CN.$ 
Therefore, we can estimate that $\|g(x,t)\|_{C^2} \leq
C(\alpha)tN^2,$ for every $0 \leq t \leq h.$
\end{proof}
From Lemma~\ref{uvrough} it follows that
\begin{eqnarray*}
\int\limits_0^h \|\th \th_x - w w_x\|_{C^1}\,ds \leq \int\limits_0^h
\left(\|(\th-w) \th_x\|_{C^1}+\|w(\th_x-w_x)\|_{C^1} \right)\,ds \leq
C(\alpha)N^3 h^2.
\end{eqnarray*}
This completes the proof of Proposition~\ref{timesplit}.
\end{proof}

The next stage is to investigate carefully a single time splitting
step. The initial data $\th_0(x)$ will be smooth, $2L$-periodic,
odd, and satisfy $\th_0(L)=0.$ It is not hard to see that all these
assumptions are preserved by the evolution. We will assume a
certain lower bound on $\th_0(x)$ for $0 \leq x \leq L,$ and derive
a lower bound that must hold after the small time step. The lower
bound will be given by the following piecewise linear functions on
$[0,L]:$
\[ \varphi (\kappa,H,a,x) = \left\{ \begin{array}{ll} \kappa x,
&  0 \leq x \leq \delta \equiv H/\kappa \\
H, & \delta \leq x \leq L-a \\
\frac{H}{a} (L-x), & L-a \leq x \leq L. \end{array} \right.
\]

Here $L,$ $\kappa,$ $H$ and $a$ may depend only on $\alpha$ and
will be specified later. We will set $a\leq L/4,$ $\delta\leq L/4$
and will later verify that this condition is preserved throughout
the construction. We assume that blow up does not happen until
time $T$ (to be determined later). Let $N = {\rm sup}_t
\|\th(x,t)\|_{C^3}.$
\begin{lemma}\label{burstep}
Assume that the initial data $\th_0(x)$ for the equation \eqref{weq}
satisfies the above assumptions. Then for every $h$ small enough
($h \leq CN^{-1}$ is sufficient), we have
\[ w(x,h) \geq \varphi \left(\frac{\kappa}{1-\kappa h}, H, a+\|\th_0\|_{L^\infty}h,x\right),\ \ \ 0\leq x\leq L. \]
\end{lemma}
\begin{proof}
The Burgers equation can be solved explicitly using
characteristics. The existence of $C^3$ solution $w(x,t)$ for $t\leq h$ is assured by the assumption on the initial data and $h.$
\end{proof}

Now we consider the effect of the viscosity time step. Suppose that the initial data $v_0(x)$ for \eqref{veq}
satisfies the same conditions as stated for $\th_0(x)$ above:
periodic, odd,
$v_0(L)=0.$ 
Then we have
\begin{lemma}\label{disstep}
Assume that for $0 \leq x \leq L,$ $v_0(x) \geq \varphi(\kappa, H,
a, x).$ Moreover, assume that
\begin{equation}\label{condstep}
H\kappa^{-1}\leq a,\ \ \ L\geq 4a,\ \ \
L^{-2\alpha}\|v_0\|_{L^\infty} \leq 4Ha^{-2\alpha}.
\end{equation}
Then for every sufficiently small $h,$ we have
\[ v(x,h) \geq \varphi(\kappa(1-C(\alpha)h H^{-2\alpha}\kappa^{2\alpha}),
H(1-C(\alpha)hH^{-2\alpha}\kappa^{2\alpha}), a,x),\ \ \ 0\leq x\leq L.
\]
\end{lemma}
The proof of this lemma is straightforward but fairly technical.
We omit the proof here, referring for details to \cite{KNS}.

Combining Proposition~\ref{timesplit} and Lemmas~\ref{burstep} and
\ref{disstep}, we obtain
\begin{theorem}\label{keystep}
Assume that the initial data $\th_0(x)$ is $2L$-periodic, odd,
$\th_0(L)=0,$ and $\th_0(x) \geq \varphi(\kappa,H,a,x).$ Suppose that
\eqref{condstep} holds with $v_0$ replaced by $\th_0.$ Assume also
that the solution $\th(x,t)$ of the equation \eqref{bur1} with
initial data $\th_0(x)$ satisfies $\|\th(x,t)\|_{C^3} \leq N$ for $0
\leq t \leq T.$ Then for every $h \leq h_0(\alpha,N)$ small
enough, we have for $0 \leq x \leq L$
\begin{equation}\label{forest}
\th(x,h) \geq \varphi(\tilde{\kappa},
\tilde{H},a+h\|\th_0\|_{L^\infty},x),
\end{equation}
where \begin{equation}\label{newkap}
\tilde{\kappa}=\kappa(1-C(\alpha)\kappa^{2\alpha}
H^{-2\alpha}h)(1-\kappa h)^{-1} - C(\alpha, N)h^2
\end{equation}
and \begin{equation}\label{newH} \tilde{H}=
H(1-C(\alpha)\kappa^{2\alpha} H^{-2\alpha}h)-C(\alpha,N)h^2.
\end{equation}
\end{theorem}
\begin{proof}
We can clearly assume that $\kappa h  \leq 1/2;$ in view of our
assumptions on $\th_0,$ $h \leq 1/2N$ is sufficient for that. Then
Lemmas~\ref{burstep} and \ref{disstep} together ensure that the
time splitting solution $v(x,h)$ of \eqref{weq} and \eqref{veq}
satisfies for $0\leq x \leq L$
\begin{equation}\label{vbound}
v(x,h) \geq \varphi(\kappa(1-C(\alpha)\kappa^{2\alpha}
H^{-2\alpha}h)(1-\kappa h)^{-1}, H(1-C(\alpha)\kappa^{2\alpha}
H^{-2\alpha}h), a+\|\th_0\|_{L^\infty}h,x).
\end{equation}
Furthermore, Proposition~\ref{timesplit} allows us to pass from
the lower bound on $v(x,h)$ to lower bound on $\th(x,h),$ leading to
\eqref{forest}, \eqref{newkap}, \eqref{newH}.
\end{proof}

From Theorem~\ref{keystep}, we immediately infer
\begin{corollary}\label{diffsys}
Under assumptions of the previous theorem and the additional
assumption stated below, for all $h$ small enough we have for $0
\leq x \leq L$ and $0 \leq nh \leq T$
\begin{equation}\label{unest}
\th(x,nh) \geq \varphi(\kappa_n, H_n, a_n ,x).
\end{equation}
Here
\begin{equation}\label{kaprec}
\kappa_n = \kappa_{n-1} ( 1-C(\alpha)\kappa_{n-1}^{2\alpha}
H^{-2\alpha}_{n-1}h)(1-\kappa_{n-1} h)^{-1} - C(\alpha, N)h^2,
\end{equation}
\begin{equation}\label{Hrec2}
H_n = H_{n-1}(1-C(\alpha)\kappa_{n-1}^{2\alpha} H^{-2\alpha}_{n-1}h)
-C(\alpha,N)h^2,
\end{equation}
and
\begin{equation}\label{arec}
a_n=a+nh\|\th_0\|_{L^\infty}.
\end{equation}
The corollary only holds assuming that for every $n,$ we have
\begin{equation}\label{Lcon}
H_n\kappa_n^{-1}\leq a_n,\ \ L\geq 4a_n,\ \
L^{-2\alpha}\|\th_0\|_{L^\infty} \leq 4H_n a_n^{-2\alpha}.
\end{equation}
\end{corollary}

To study \eqref{kaprec} and \eqref{Hrec2}, we introduce the
following system of differential equations:
\begin{equation}\label{diffeqsys}
\kappa' = \kappa^2 - C(\alpha) \kappa^{1+2\alpha}
H^{-2\alpha};\,\,\,H'=-C(\alpha)\kappa^{2\alpha} H^{1-2\alpha}.
\end{equation}
\begin{lemma}\label{deapp}
Assume that $[0,T]$ is an interval on which the solutions of the
system \eqref{diffeqsys} satisfy $|\kappa (t)| \leq 2N,$
$0<H_1(\alpha) \leq H(t) \leq H_0(\alpha).$ Then for every
$\epsilon>0,$ there exists $h_0(\alpha,N,\epsilon)>0$ such that if
$h<h_0,$ then $\kappa_n$ and $H_n$ defined by \eqref{kaprec} and
\eqref{Hrec2} satisfy $|\kappa_n - \kappa(nh)|<\epsilon,$ $|H_n
-H(nh)| < \epsilon$ for every $n \leq [T/h].$
\end{lemma}
\begin{proof}
This is a standard result on approximation of differential
equations by a finite difference scheme. Observe that the
assumptions on $\kappa(t)$ and $H(t)$ also imply upper bounds on
$\kappa'(t),$ $\kappa''(t),$ $H'(t)$ and $H''(t)$ by a certain
constant depending only on $N$ and $\alpha.$ The result can be
proved comparing the solutions step-by-step inductively. Each step
produces an error not exceeding $C_1(\alpha,N)h^2,$ and the total
error over $[T/h]$ steps is estimated by $C_1(\alpha,N)h.$
Choosing $h_0(\alpha,N,\epsilon)$ sufficiently small completes the
proof.
\end{proof}
The final ingredient we need is the following lemma on the
behavior of solutions of the system \eqref{diffeqsys}.
\begin{lemma}\label{conlaw}
Assume that the initial data for the system \eqref{diffeqsys}
satisfy
\begin{equation}\label{iccon} H_0^{2\alpha} \kappa_0^{1-2\alpha}
\geq C(\alpha)/(1-2\alpha). \end{equation} Then on every interval
$[0,T]$ on which the solution makes sense (that is, $\kappa(t)$
bounded), the function $H(t)^{2\alpha} \kappa^{1-2\alpha}(t)$ is
non-decreasing.
\end{lemma}
\begin{proof}
A direct computation shows that
\[ \left( H(t)^{2\alpha} \kappa^{1-2\alpha}(t) \right)' = (1-2\alpha)
\kappa(t) \left(H(t)^{2\alpha} \kappa(t)^{1-2\alpha} -
\frac{C(\alpha)}{1-2\alpha}\right). \]
\end{proof}
Now we are ready to complete the blow up construction.
\begin{proof}[Proof of Theorem~\ref{burbu}]
Set $\kappa_0$ to be large enough, in particular
\begin{equation}\label{Hzero}
\kappa_0 = \left(\frac{3C(\alpha)}{1-2\alpha}
\right)^{\frac{1}{1-2\alpha}}
\end{equation}
will do. Set $H_0=1$, $a=\kappa_0^{-1}$,
$T(\alpha)=\frac{3}{2\kappa_0}$. Choose $L$ so that
\begin{equation}\label{Lchoice}
L \geq 16a.
\end{equation}
The initial data $\th_0(x)$ will be a smooth, odd, $2L-$periodic
function satisfying $\th_0(L)=0$ and $\th_0(x) \geq
\varphi(\kappa_0,H_0,a,x).$ We will also assume
$\|\th_0\|_{L^\infty} \leq 2H_0.$ Observe that $H_0$ and $\kappa_0$
are chosen so that in particular the condition \eqref{iccon} is
satisfied.
From
\eqref{diffeqsys} and Lemma~\ref{conlaw} it follows that
\begin{equation}\label{kappabu}
\kappa' = \kappa^2 -C(\alpha)\kappa^{1+2\alpha}H^{-2\alpha} \geq
\frac23\kappa^2.
\end{equation}
This implies $\kappa (t) \geq \frac{1}{\kappa_0^{-1}-\frac23 t}.$ In
particular, there exists $t_0 < T(\alpha)$ such that
$\kappa(t_0)=2N$ for the first time. Note that due to
\eqref{kappabu}, for $0 \leq t \leq t_0$ we have
\begin{equation}\label{kup}
 \kappa(t) \leq \frac{1}{\frac23(t_0-t)+\frac{1}{2N}} \leq
\frac{3/2}{(t_0-t)}.
\end{equation}
Rewrite the equation for $H(t)$ as
\begin{equation}\label{Halpha}
(H^{2\alpha})' = -2C(\alpha)\alpha \kappa^{2\alpha}.
\end{equation}
Using the estimate \eqref{kup} in \eqref{Halpha}, we get that for
any $0 \leq t \leq t_0,$
\[ H^{2\alpha}(t) \geq H_0^{2\alpha} - 2C(\alpha)\alpha
\int_0^{t_0}\kappa^{2\alpha}(s)\,ds \geq H_0^{2\alpha}(1-\alpha).
\] We used the fact that $H_0=1,$ $t_0 < T(\alpha) = \frac{3}{2\kappa_0}$
and \eqref{Hzero}. Now we can apply Lemma~\ref{deapp} on the
interval $[0,t_0].$ Choosing $\epsilon$ and $h$ sufficiently
small, we find that for $0 \leq nh \leq t_0,$ $\kappa_n \geq 1$
and $H_n \geq (1-\alpha)^{1/2\alpha}H_0\geq H_0/2.$ Also,
evidently, $a_n \leq a+2H_0 T(\alpha)=4a.$ This allows us to check
that the conditions \eqref{Lcon} hold on each step due to the
choice of $L$ \eqref{Lchoice}, justifying control of the true PDE
dynamics by the system \eqref{diffeqsys}.

From Lemma~\ref{deapp} and $\kappa(t_0)=2N$, we also see that,
given that $h$ is sufficiently small, $\kappa_{n_0} \geq 3N/2$ for
some $n_0$ such that $n_0h \leq t_0 < T(\alpha).$
Thus Corollary~\ref{diffsys} provides us with a lower bound
$\th(0,n_0 h)=0,$ $\th(x, n_0 h) \geq 3Nx/2$ for small enough $x.$
This contradicts our assumption that $\|\th(x,t)\|_{C^3} \leq N$ for
$0 \leq t \leq T(\alpha),$ thus completing the proof.
\end{proof}

We obtained blow up in the case where period $2L$ was sufficiently
large (depending only on $\alpha$). However, examples of blow up
with arbitrary periodic data follow immediately from a scaling
argument. Indeed, assume $\th(x,t)$ is a $2L-$periodic solution of
\eqref{bur1}. Then $\th_1(x,t) = L^{-1+2\alpha} \th(Lx,L^{2\alpha} t)$
is a $2-$periodic solution of the same equation. Thus a scaling
procedure allows to build blow up examples for any period.

\it Remark. \rm Formally we proved the blow up only in $C^3$
class. But global regularity in $H^s$ class for
$s>\frac32-2\alpha$ provides global regularity in $C^\infty$ (see
\cite{KNS}, Theorem 1, so that blow up happens
in every $H^s$ class for
$s>\frac32-2\alpha$. \\

\vspace*{0.5cm}
\setcounter{equation}{0}
\section{Finite time blow up: the supercritical CCF model}\label{buccf}

In this section, we provide an elementary argument proving finite time blow up
in the CCF model for the part of supercritical dissipation range $1/4>\alpha  \geq 0.$ We will consider
the whole line setting here, since the argument is less technical in this case. Recall that the CCF equation is given by
\begin{equation}\label{CCFeq11}
\theta_t = (H\theta) \theta_x -(-\Delta)^\alpha \theta, \,\,\,\theta(x,0)=\theta_0(x),
\end{equation}
where $H\theta$ is the Hilbert transform of $\th.$

\begin{theorem}\label{CCFbu}
Let $1/4 > \alpha \geq 0.$ There exist smooth initial data $\theta_0(x)$ such that the solution $\theta(x,t)$ to the
dissipative CCF equation
forms a singularity in finite time.
\end{theorem}

The blow up picture for the CCF equation is quite different from Burgers. If one takes positive, even data  with
a maximum value at $x=0,$ which is sufficiently large in an appropriate sense, then the solution develops a cusp
singularity at $x=0$ in finite time.

Finite time blow up for $\alpha =0$ was first proved by Cordoba, Cordoba and Fontelos \cite{CCF},
and the result was generalized to $1/4 >\alpha >0$ range by Li and Rodrigo. Both arguments are based
on an ingenious inequality: if $f \in C_0^\infty(\Rm),$ is even, $f(0)=0,$ and $Hf$ is the Hilbert transform of $f$
defined by \eqref{ccfas}, then for every $0<\delta<1,$ we have
\begin{equation}\label{htineq11}
-\int_0^\infty \frac{f_x(x)(Hf)(x)}{x^{1+\delta}}\,dx \geq C_\delta \int_0^\infty \frac{f(x)^2}{x^{2+\delta}}\,dx.
\end{equation}

Let us first sketch how one gets finite time blow up given \eqref{htineq11}.
\begin{proof}[Proof of Theorem~\ref{CCFbu}]
Let us take even, positive initial data $\theta_0(x)$
with maximum at $x=0.$ Assume, on the contrary, that the solution
stays smooth for all times.  Fix $\delta$ with $0<\delta<1-4\alpha.$ Consider
\[     J(t) \equiv \int_0^\infty \frac{\theta(0,t)-\theta(x,t)}{x^{1+\delta}}\,dx.     \]
Notice that if $\delta >0,$ blow up in $J(t)$ implies loss of regularity by $\th$ at $x=0.$
Then
\begin{equation}\label{Jest}J'(t) = -\int_0^\infty \frac{\theta_x(x,t)(H\theta)(x,t)}{x^{1+\delta}}\,dx
-\int_0^\infty \frac{((-\Delta)^\alpha \theta) (0,t)-((-\Delta)^\alpha \theta) (x,t)}{x^{1+\delta}}\,dx. \end{equation}
According to \eqref{htineq11},
\begin{equation}\label{htineq12}  -\int_0^\infty \frac{\theta_x(x,t)(H\theta)(x,t)}{x^{1+\delta}}\,dx \geq C_\delta \int_0^\infty
\frac{(\theta(0,t)-\theta(x,t))^2(x)}{x^{2+\delta}}\,dx. \end{equation}
Also, recall that (see e.g. \cite{CC}) if $0<\beta<1/2,$
\begin{equation}\label{fracder}
((-\Delta)^\beta f)(y) = P.V.\int_{-\infty}^\infty \frac{f(x)-f(y)}{x^{1+2\beta}}\,dx. \end{equation}
Obviously, $(-\Delta)^{\delta/2}((-\Delta)^\alpha f)|_{x=0} = (-\Delta)^{\alpha+\delta/2}f|_{x=0}.$ Applying this identity
and \eqref{fracder} to the last term in \eqref{Jest}, and using that $\theta$ is even, we obtain that
\begin{equation}\label{fracest11} \int_0^\infty \frac{((-\Delta)^\alpha \theta) (0,t)-((-\Delta)^\alpha \theta) (x,t)}{x^{1+\delta}}\,dx
= \int_0^\infty \frac{\theta(0,t)-\theta(x,t)}{x^{1+\delta+2\alpha}}\,dx. \end{equation}
Next, we need the following simple lemma.
\begin{lemma}\label{fracestlem}
Assume $\delta$ satisfies $0 <\delta < 1.$ Assume $f(x) \in C^1(\Rm^+) \cap L^\infty(\Rm^+)$ and $f(0)=0.$ Then
\begin{equation}\label{best21}
\int_0^\infty \frac{f(x)^2}{x^{2+\delta}}\,dx \geq C_1 \left( \int_0^\infty \frac{|f(x)|}{x^{1+\delta}}\,dx \right)^2 - C_2 (1+\|f\|^2_{L^\infty}).
\end{equation}
Also, for every $0 < \alpha < 1/4,$ $0<\delta <1-4\alpha,$ and for every $\epsilon >0,$ we have
\begin{equation}\label{best22}
\int_0^\infty \frac{f(x)}{x^{1+\delta+2\alpha}}\,dx \leq \epsilon \int_0^\infty \frac{f(x)^2}{x^{2+\delta}}\,dx + C_\epsilon (1+\|f\|_{L^\infty}).
\end{equation}
\end{lemma}
The lemma can be proved by application of H\"older inequality. The $\alpha <1/4$ restriction comes from failure of \eqref{best22} to hold for any
$\delta>0$ if $\alpha \geq 1/4.$
Given \eqref{htineq11}, \eqref{fracest11}, \eqref{best21} and \eqref{best22}, \eqref{Jest} leads to
\[  J'(t) \geq C_3 J(t)^2 - C_4 (1+\|\theta_0\|^2_{L^\infty}), \]
leading to blow in $J(t)$ in finite time. But this can only happen if the solution $\theta(x,t)$ loses regularity at $x=0.$
\end{proof}

Finally, let us present the proof of \eqref{htineq11} which is less general than in \cite{CCF} but elementary and suffices for our application.
First, we need a monotonicity result.
\begin{lemma}\label{ccfmon}
Assume that the initial data $\theta_0(x)$ is smooth, bounded, positive, even, and monotone decaying on $(0,\infty).$ Then while the solution $\theta(x,t)$
remains smooth, it stays positive, even and monotone decaying on $(0,\infty).$
\end{lemma}
\begin{proof}
We only need to comment on the preservation of the decay property. The boundedness and positivity follow from the standard maximum principle (e.g. \cite{CC}),
while the preservation of evenness is easy to check. The equation for the derivative of $\theta$ is given by
\begin{equation}\label{mpder12}
\partial_t(\theta_x) = (H\theta)_x \theta_x +  H\theta (\theta)_{xx} - (-\Delta)^\alpha \theta_x, \,\,\,\theta_x(0,t)=0,\,\,\,\forall \,t.
\end{equation}
Now the property $\theta_x(x,0)>0$ for $x>0$ is preserved for all times by an argument similar to the usual fractional diffusion maximum principle (see e.g. \cite{CC}).
Informally, if $t_1$ is the first time and $x_1>0$ is a point where $\theta_x(x_1,t_1)=0,$ then the first two terms on the right hand side of \eqref{mpder12}
vanish at $x_1,$ while the last term gives a positive contribution (consider \eqref{fracder}).
\end{proof}

The proof of \eqref{htineq11} that we present below works only for functions that are positive, even, and monotone decaying on $(0,\infty).$ However, due to Lemma~\ref{ccfmon},
these properties are preserved by evolution. Therefore, we can prove the finite time blow up using the same argument as before.
The proposition below should be applied to $\theta(0,t)-\theta(x,t) = f(x),$ with $p=1$ and $\sigma = 1+\delta,$ $0<\delta<1.$

\begin{proposition}\label{nonlinlp}
Assume that function $f(x)$ is $C^1,$ even, $f'(x)\geq 0$ for $x >
0,$ $f$ is bounded on $\Rm,$ and $f(0)=0.$ Then
\begin{equation}\label{mainnonlin}
-\int_0^1 \frac{Hf(x) f'(x) f(x)^{p-1}}{x^{\sigma}}\,dx \geq C_0
\int_0^1 \frac{f(x)^{p+1}}{x^{1+\sigma}}\,dx,
\end{equation}
for any $p \geq 1$ and any $\sigma>0.$ The constant $C_0$ may
depend only on $p$ and $\sigma.$ If the right hand side of
\eqref{mainnonlin} is infinite, the inequality is understood in
the sense that the left hand side must also be infinite.
\end{proposition}
\it Remarks. \rm 1. In \cite{CCF}, the authors show the inequality
\eqref{mainnonlin} in the case of $p=1,$ $1<\sigma<2$ and general
even (not necessarily monotone) $f.$ We do not know if our
inequality still holds in this broader generality. \\
2. In our application, divergent integrals in \eqref{mainnonlin}
imply that blow up has already happened. Thus we do not need to
consider this case but include it for completeness. \\

The first step is the following
\begin{lemma}\label{htest}
If $f(x) \in C^1$ and is even, the following representation holds:
\begin{equation}\label{htesteq}
Hf(x) = \int_0^{2x} \log \left| \frac{y-x}{x} \right| f'(y)\,dy +
\int_x^\infty \frac{f(y-x)-f(y+x)}{y}\,dy.
\end{equation}
\end{lemma}
\begin{proof}
Direct computation using integration by parts for the first term.
\end{proof}

\begin{corollary}\label{htcor}
 Suppose that $f(x)$ satisfies the
assumptions of Proposition~\ref{nonlinlp}. Then for any $1<q<2$ we
have \begin{equation}\label{htest11} Hf(x) \leq \log(q-1)
\int_{q^{-1}x}^{qx}f'(y)\,dy = \log(q-1) (f(qx)-f(q^{-1}x)).
\end{equation}
\end{corollary}
\begin{proof}
Follows immediately from \eqref{htesteq} and monotonicity of $f(x)$
for $x>0.$
\end{proof}

\begin{proof}[Proof of Proposition~\ref{nonlinlp}]
Fix a number $q,$ $1<q<2.$ Also fix $c>0$ such that
$(1-c)^{-p-1}q^{-\sigma}<1.$ Corollary~\ref{htcor} reduces the
proof to the proof of the following inequality
\begin{equation}\label{mainineq1}
\int_0^1 \frac{f'(x) (f(qx)-f(q^{-1}x))f(x)^{p-1}}{x^{\sigma}}\,dx
\geq C_0 \int_0^1 \frac{f(x)^{p+1}}{x^{1+\sigma}}\,dx
\end{equation}
with some $C_0>0$ which may depend only on $p$ and $\sigma.$ Split
integration on both sides of the inequality into intervals
$[q^{-n-1},q^{-n}],$ and set $a_n \equiv f(q^{-n}).$ Notice that
\begin{equation} s_n \equiv \int_{q^{-n-1}}^{q^{-n}}\frac{f'(x)
(f(q x)-f(q^{-1}x))f(x)^{p-1}}{x^{\sigma}}\,dx \geq
(a_n-a_{n+1})\int_{q^{-n-1}}^{q^{-n}}\frac{f'(x)
f(x)^{p-1}}{x^{\sigma}}\,dx. \label{rhest3}
\end{equation}
Let us integrate by parts in the last integral in \eqref{rhest3}:
\begin{equation}\label{intest17}
\int_{q^{-n-1}}^{q^{-n}}\frac{f'(x) f(x)^{p-1}}{x^{\sigma}}\,dx =
\frac1p (a_n^p - a_{n+1}^p) q^{\sigma n} + \frac{\sigma}{p}
\int_{q^{-n-1}}^{q^{-n}}
\frac{f(x)^p-a_{n+1}^p}{x^{1+\sigma}}\,dx.
\end{equation}
Also, \begin{equation}\label{rhsest14} r_n \equiv
\int_{q^{-n-1}}^{q^{-n}} \frac{f(x)^{p+1}}{x^{1+\sigma}}\,dx \leq
a_n^{p+1}q^{\sigma(n+1)}.
\end{equation}
Let us call $n$ "good" if $a_n-a_{n+1} \geq ca_n$ (recall $c$ is
such that $(1-c)^{-p-1}q^{-\sigma}<1$), and "bad" otherwise.
\begin{lemma}\label{goodnlemma}
If the set of all good $n$ is finite, then both sides of
\eqref{mainnonlin} are infinite.
\end{lemma}
\begin{proof}
Assume that the set of good $n$ is finite. Then there exists $N$
such that for all $n>N,$ we have $a_{n+1}>(1-c)a_n.$ But then
\[ r_n \geq a_{n+1}^{p+1}q^{n(1+\sigma)}(q^{-n}-q^{-n-1}) \geq
(1-q^{-1})a_N^{p+1}(1-c)^{(n+1-N)(p+1)}q^{n(1+\sigma)} \rightarrow
\infty \] as $n \rightarrow \infty$ by our choice of $c.$ So the
right hand side of \eqref{mainnonlin} diverges.

To show that the left hand side diverges as well, let $c_n$ be
such that $a_n-a_{n+1} = c_n a_n,$ $c_n <c$ if $n>N.$ Since $a_n
\rightarrow 0$ as $n \rightarrow \infty,$ we must have
$\prod_{n=N}^\infty (1-c_n) =0,$ implying $\sum_n c_n = \infty.$
On the other hand, \[ s_n \geq (a_n-a_{n+1})^2
a_{n+1}^{p-1}q^{\sigma n} \geq c_n^2 a_{n+1}^{p+1}q^{\sigma n}. \]
Under our assumptions, the expression $a_{n+1}^{p+1}q^{\sigma n}$
is bounded from below by increasing geometric progression. This
along with $\sum_n c_n = \infty$ implies easily that $\sum_n s_n$
also diverges.
\end{proof}

Thus we can assume that the set of good $n$ is infinite. If $n$ is
good, then, by \eqref{rhest3}, \eqref{intest17}, \eqref{rhsest14}
and monotonicity of $f(x),$
\[ s_n \geq \frac{ca_n}{p} (a_n^p -a_{n+1}^p)q^{\sigma n} \geq
\frac{c^2}{p} a_n^{p+1}q^{\sigma n} \geq \frac{c^2}{p
q^\sigma}r_n. \] Suppose that $(n_{j-1},n_j)$ is an interval of
bad $n,$ while $n_{j-1}$ and $n_j$ are good. Since for bad $n$ we
have $a_n < \frac{1}{1-c}a_{n+1},$ for every $n \in (n_{j-1},n_j)$
the following estimate holds:
\[ r_n \leq a_n^{p+1}q^{\sigma (n+1)} \leq
\left(\frac{1}{1-c}\right)^{(n_j-n)(p+1)}q^{\sigma
(n-n_j+1)}a_{n_j}^{p+1}q^{\sigma n_j} \leq \frac{p q^\sigma}{c^2}
\left(\frac{1}{(1-c)^{p+1}q^\sigma}\right)^{n_j-n}s_{n_j}. \] Thus
\[ \sum_{n=n_{j-1}+1}^{n_j} r_n \leq \frac{p q^\sigma}{c^2}s_{n_j}
\sum_{k=0}^{n_j-n_{j-1}}\left(\frac{1}{(1-c)^{p+1}q^\sigma}\right)^k
\leq \frac{p (1-c)^{p+1}q^{2\sigma}}{c^2((1-c)^{p+1}q^\sigma-1)}
s_{n_j}. \] Summing over all $n$ and estimating all "bad"
intervals as above, we obtain \eqref{mainnonlin}.
\end{proof}

\vspace*{0.5cm}
\setcounter{equation}{0}
\section{Conclusions and some open problems}

In this review, we focused on the questions of global existence and regularity or blow up for
active scalars. These equations constitute a natural class of models exhibiting
rich and intricate behaviors, yet appearing more approachable than, for example,
three dimensional Euler or Navier-Stokes equations. There is a number of interesting open problems,
varying in difficulty, that we discussed in this review. For the convenience of the reader, we list some of these
problems below. \\

1. Prove global regularity or finite time blow up for CCF equation in $1/4 \leq \alpha < 1/2$ regime. \\

The result would be more interesting if it is global regularity. In this case, there must be some nonlinearity
depletion in the equation that is not captured by scaling. This would be a step towards understanding this phenomena
better and trying to project this understanding to more complex equations, such as for example SQG. \\

2. Prove or disprove uniqueness of solutions to Burgers or SQG equation with rough ($L^p$) initial data. \\

This question sounds technical, but the issues one has to overcome are likely to provide better insight into the
inner workings of these equations, on a sufficiently subtle level. \\

3. Prove global regularity of solutions or finite time blow up for the SQG equation in the regime $0 \leq \alpha < 1/2.$ \\

This is a major open problem that has attracted much attention over the recent years, but progress towards this goal has
been very limited. Major new ideas will likely be needed. \\

As far as trying to prove finite time blow up for the SQG equation, a more modest and reasonable goal is to show infinite
growth of gradient or other higher order Sobolev norm of the solution. No such results are known for the SQG equation.
Thus we can ask the following natural question. \\

4. Build examples where solutions of the SQG equation show infinite growth of gradient or other higher order Sobolev norm as time goes to infinity. \\

The solutions of the SQG equation are often quite unstable, and thus difficult to control. To prove infinite growth, one has to construct
solutions with "stable instability".     \\

The question about infinite growth in time of the gradient of solution is also very interesting for 2D Euler equation in vorticity form, the most classical example of
active scalar. One has upper bound by double exponential, a consequence of the conservation of $\|\theta\|_{L^\infty}$ (see e.g. \cite{BM}).
From the opposite side, the best known result is superlinear growth due to Denisov \cite{Den1} (see also earlier works by Yudovich \cite{Y,Y1} and Nadirashvili \cite{Nad}).
Hence we finish our list with perhaps the longest-open question.  \\

5. Close or shrink the gap between the upper bounds and actual growth in examples of the gradient of vorticity (or some other higher order Sobolev norm)
of the solutions of 2D Euler equation. \\

\vspace*{0.5cm}

\section*{Acknowledgements}
This work has been partially supported by the
NSF-DMS grant 0653813. The author expresses his gratitude to the
Department of mathematics of the University of Chicago, where part of this
work was carried out.


\end{document}